\documentclass[11pt, a4paper]{amsart}

\usepackage[T1]{fontenc}      
\usepackage{newtxtext}

\usepackage{amsmath,amssymb,amsthm,mathtools}
\usepackage{enumitem}
\usepackage{hyperref}
\usepackage[normalem]{ulem} 
\numberwithin{equation}{section}
\newtheorem{theorem}{Theorem}[section]
\newtheorem{cor}[theorem]{Corollary}
\newtheorem{lemma}[theorem]{Lemma}

\theoremstyle{definition} 
\newtheorem{definition}[theorem]{Definition}

\newcommand{\N}{\mathbb{N}}
\newcommand{\coz}{\operatorname{coz}}
\newcommand{\supp}{\operatorname{supp}}

  \usepackage{marginnote}
 \usepackage{geometry,ulem}

\begin{document}

\title{Order Isomorphisms between Positive Cones of \(C_0(X)\)}

\author[N.~Shibata]{Natsumi Shibata}
\address[N. Shibata]{Graduate School of Science and Technology,
Niigata University,
Niigata 950-2181, Japan}
\email{f25a056h@mail.cc.niigata-u.ac.jp}

\author[I.~Matsuzaki]{Izuho Matsuzaki}
\address[I.~Matsuzaki]
{Graduate School of Science and Technology,
Niigata University, Niigata 950-2181, Japan}
\email{matsuzaki@m.sc.niigata-u.ac.jp}

\author[T. Miura]{Takeshi Miura}
\address[T. Miura]
{Department of Mathematics,
Faculty of Science,
Niigata University,
Niigata 950-2181, Japan}
\email{miura@math.sc.niigata-u.ac.jp}

\subjclass{Primary 46J10; Secondary 46B20}
\keywords{order isomorphism, positive cone,
weighted composition operator}

\begin{abstract}
Let \(X\) and \(Y\) be locally compact Hausdorff spaces.
We study order isomorphisms
\[
T:C_0^+(X)\to C_0^+(Y),
\]
where \(C_0(X)\) denotes the Banach space of all
real-valued continuous functions on \(X\) vanishing at infinity,
and
\[
C_0^+(X)=\{f\in C_0(X):f\ge0\}
\]
is its positive cone.

We assume that
$T$ is positive homogeneous. That is,
\[
T(rf)=rT(f)
\qquad (r>0,\,f\in C_0^+(X)).
\]
Under this assumption, we prove that \(T\) is represented
as a weighted composition operator induced by a homeomorphism
from \(Y\) onto \(X\) and a bounded continuous weight function.
Moreover, we show that \(T\) extends uniquely to a linear
order isomorphism between \(C_0(X)\) and \(C_0(Y)\).
\end{abstract}

\maketitle

\section{Introduction and Main Theorem}

In recent years, order-preserving mappings
on positive cones have been actively studied
by Moln\'{a}r and others
(see \cite{dong,gao,hato,lemm,moln}).
Let \(C(X)\) denote the space of all real-valued
continuous functions on a compact Hausdorff space
\(X\), and let
\[
C(X)_{++}=\{f\in C(X): f>0\}.
\]
In 1979, Sch\"affer  investigated order isomorphisms
from \(C(X)_{++}\) onto \(C(Y)_{++}\). Here, a bijection
\[
T:C(X)_{++}\to C(Y)_{++}
\]
is called an order isomorphism if
\[
f<g
\quad\Longleftrightarrow\quad
T(f)<T(g)
\qquad (f,g\in C(X)_{++}).
\]
Assuming the existence of a function
\(w\in C(X)_{++}\) satisfying
\[
T(rw)=rT(w)
\qquad (r>0),
\]
Sch\"affer proved in \cite[Corollary~8.4]{scha}
that \(T\) is a weighted composition operator
induced by a homeomorphism from \(Y\) onto \(X\).
Consequently, when \(X\) and \(Y\) are compact,
the topological structure of the underlying spaces
is completely determined by order isomorphisms
between the corresponding positive cones.

In the theory of \(C^*\)-algebras, one distinguishes the
positive definite cone, consisting of positive invertible elements,
from the positive semidefinite cone, consisting of all positive elements.
Many structural results on positive homogeneous order isomorphisms
concern positive definite cones. In the commutative unital case,
positive definite cones correspond to \(C(X)_{++}\), the cone of strictly positive
continuous functions. In the present paper we deal instead with the
positive semidefinite cone \(C_0^+(X)\) in the non-unital commutative
setting. This distinction is essential, since \(C_0^+(X)\) contains
functions with zeros and, in general, has no strictly positive order unit.

In contrast, for locally compact Hausdorff spaces
\(X\) and \(Y\), we consider order isomorphisms
\(T:C_0^+(X)\to C_0^+(Y)\).
Here, \(C_0(X)\) denotes the Banach space
of all real-valued continuous functions on \(X\)
vanishing at infinity, and
\[
C_0^+(X)=\{f\in C_0(X): f\ge 0\}
\]
is its positive cone. A mapping
\(T:C_0^+(X)\to C_0^+(Y)\)
is called an order isomorphism
if \(T\) is bijective and satisfies
\[
f\le g
\quad\Longleftrightarrow\quad
T(f)\le T(g)
\qquad (f,g\in C_0^+(X)).
\]
To the best of the authors' knowledge,
it is unknown whether a Sch\"affer-type
theorem holds for order isomorphisms
between \(C_0^+(X)\) and \(C_0^+(Y)\).

When \(X\) is locally compact and non-compact,
the existence of a strictly positive function \(w\)
satisfying Sch\"affer's assumption is no longer guaranteed.
In this paper, we consider positive homogeneous order isomorphisms.
More precisely, we assume that
\[
T(rf)=rT(f)
\qquad (r>0,\ f\in C_0^+(X)).
\]
This homogeneity assumption is stronger than Sch\"affer's
one-ray condition. On the other hand, our setting is different
in an essential way: we consider the positive semidefinite cone
consisting of nonnegative functions, rather than the cone of
strictly positive functions.

Our main theorem is as follows.

\begin{theorem}\label{thm:main}
Let \(X\) and \(Y\) be locally compact
Hausdorff spaces, and let
\(T:C_0^+(X)\to C_0^+(Y)\)
be a positive homogeneous order isomorphism.
Then there exist a constant \(\delta>0\),
a unique bounded continuous function
\(\alpha:Y\to[\delta,\infty)\),
and a unique homeomorphism
\(\tau:Y\to X\)
such that
\begin{equation}\label{eq:T}
T(f)(y)=\alpha(y)f(\tau(y))
\end{equation}
for all \(f\in C_0^+(X)\) and \(y\in Y\).

Conversely, every mapping of the form
\eqref{eq:T} is a
positive homogeneous
 order isomorphism.
\end{theorem}

As a corollary of the main theorem,
every
positive homogeneous
 order isomorphism on the positive cone
extends uniquely to a linear order isomorphism
on the whole space.

\begin{cor}\label{cor:extension}
Let \(T:C_0^+(X)\to C_0^+(Y)\) be a
positive homogeneous
 order isomorphism.
Then \(T\) extends uniquely to a linear order isomorphism
\[
\widetilde T:C_0(X)\to C_0(Y).
\]
\end{cor}

In particular, the positive cone \(C_0^+(X)\)
completely determines the linear order structure
of \(C_0(X)\).

The key idea of the proof is to construct
a correspondence between the underlying spaces
by analyzing intersections of families of zero sets
\(\bigcap Z(f)\).
To achieve this, we first establish
the non-emptiness of intersections of supports
\(\bigcap \supp(f)\).
A crucial role is played by the truncations
\((f-1/n)^+\),
which allow us to obtain compactly supported
functions even in the locally compact setting.
This makes it possible to adapt techniques
that are available in the compact case.

To prove that the order isomorphism
\(T\) is a weighted composition operator,
we exploit the positive homogeneity assumption
to construct the weight function
and thereby determine the precise form of \(T\).

Section~2 begins with the necessary definitions
and notation.
We then establish several auxiliary lemmas
and finally prove Theorem~\ref{thm:main}
and Corollary~\ref{cor:extension}.

\section{Proof of the Main Theorem}

\subsection{Preliminaries}

In the rest of
this manuscript, \(X\) and \(Y\) are
non-empty locally compact Hausdorff spaces.
We denote by \(C_0(X)\) the Banach space
of all real-valued continuous functions on \(X\)
which vanish at infinity,
equipped with the supremum norm \(\|\cdot\|\).
The positive cone of \(C_0(X)\) is defined by
\[
 C_0^+(X)=\{f\in C_0(X): f(x)\ge 0\ \text{for all }x\in X\}.
\]
For \(f,g\in C_0^+(X)\), we write
\[
 f\le g
 \quad\Longleftrightarrow\quad
 f(x)\le g(x)\qquad (x\in X).
\]
With respect to this pointwise order,
\(C_0^+(X)\) is a lattice.
Namely, for \(f,g\in C_0^+(X)\), we put
\[
 (f\wedge g)(x)=\min\{f(x),g(x)\},\quad
 (f\vee g)(x)=\max\{f(x),g(x)\}\qquad
 (x\in X).
\]

For \(f\in C_0^+(X)\), its cozero set and support are denoted by
\[
 \coz(f)=\{x\in X:f(x)>0\},\qquad
 \supp(f)=\overline{\coz(f)},
\]
where the closure is taken in \(X\).
For \(f\in C_0(X)\)
and \(n\in\N\), we use the notation
\[
 \left(f(x)-\frac{1}{n}\right)^+
 =\max\left\{f(x)-\frac{1}{n},0\right\}
 \qquad (x\in X).
\]
Then \((f-1/n)^+\in C_0^+(X)\).

In the rest of this section,
let \(T:C_0^+(X)\to C_0^+(Y)\)
be a
positive homogeneous
order isomorphism.
Then $T^{-1}:C_0^+(Y)\to C_0^+(X)$ is also
a positive homogeneous order isomorphism.

\subsection{Basic properties of the order isomorphism}

Our first goal is to construct a correspondence
between the underlying spaces.
To this end, we investigate
intersections of zero sets
associated with certain families of functions.
The following lemma is the first step in this direction.

\begin{lemma}\label{lem:lattice}
The order isomorphism \(T\) preserves
lattice operations:
\[
 T\left(\bigwedge_{i=1}^n f_i\right)
 =
 \bigwedge_{i=1}^n T(f_i),\qquad
 T\left(\bigvee_{i=1}^n f_i\right)
 =
 \bigvee_{i=1}^n T(f_i),
\]
for every finite family \(f_1,\ldots,f_n\in C_0^+(X)\).
\end{lemma}

\begin{proof}
Let \(f,g\in C_0^+(X)\).
We first prove \(T(f\wedge g)=T(f)\wedge T(g)\).
Since \(f\wedge g\le f\) and \(f\wedge g\le g\),
the order-preserving property of \(T\) gives
\(T(f\wedge g)\le T(f)\)
and \(T(f\wedge g)\le T(g)\).
Hence
\[
 T(f\wedge g)\le T(f)\wedge T(g).
\]
Conversely, since \(T(f)\wedge T(g)\le T(f)\)
and \(T(f)\wedge T(g)\le T(g)\),
the order-preserving property of \(T^{-1}\) gives
\( T^{-1}(T(f)\wedge T(g))\le f\) and
\(T^{-1}(T(f)\wedge T(g))\le g\).
Thus
\[
 T^{-1}(T(f)\wedge T(g))\le f\wedge g.
\]
Applying \(T\), we obtain
\[
 T(f)\wedge T(g)
\le T(f\wedge g).
\]
Therefore \(T(f\wedge g)=T(f)\wedge T(g)\).

The proof for \(\vee\) is analogous.
Since \(f\le f\vee g\) and \(g\le f\vee g\), we have
\(T(f)\vee T(g)\le T(f\vee g)\).
Conversely, from \(T(f)\le T(f)\vee T(g)\)
and \(T(g)\le T(f)\vee T(g)\), we obtain
\[
 f\le T^{-1}(T(f)\vee T(g)),\qquad
 g\le T^{-1}(T(f)\vee T(g)).
\]
Hence
\(f\vee g\le T^{-1}(T(f)\vee T(g))\).
Applying \(T\), we get
\(T(f\vee g)\le T(f)\vee T(g)\).
Thus \(T(f\vee g)=T(f)\vee T(g)\).

The general case follows by induction.
\end{proof}

We next establish a basic property
of order isomorphisms.
Although elementary, it will play
an important role in the subsequent arguments.

\begin{lemma}\label{lem:zero_disjoint}
The following assertions hold:
\begin{enumerate}[label=(\roman*)]
\item \(T(0)=0\);
\item for all \(f,g\in C_0^+(X)\),
\[
 f\wedge g=0
 \quad\Longleftrightarrow\quad
 T(f)\wedge T(g)=0.
\]
\end{enumerate}
\end{lemma}

\begin{proof}
The zero function is the least element
of \(C_0^+(X)\). Hence \(0\le f\) for every \(f\in C_0^+(X)\).
Since \(T\) preserves the order,
we obtain \(T(0)\le T(f)\) for every \(f\in C_0^+(X)\).
Since \(T\) is surjective,
\(T(0)\) is the least element of \(C_0^+(Y)\).
Therefore \(T(0)=0\), which proves the
assertion (i).

By Lemma~\ref{lem:lattice}, \(T(f\wedge g)=T(f)\wedge T(g)\). Hence
\[
 f\wedge g=0
 \quad\Longleftrightarrow\quad
 T(f\wedge g)=T(0)
 \quad\Longleftrightarrow\quad
 T(f)\wedge T(g)=0.
\]
This proves the assertion (ii).
\end{proof}

To analyze the supports of functions in \(F_y\) in the subsequent arguments,
we will utilize truncations of functions. Controlling the behavior of
these truncations requires the boundedness of \(T\), which we establish next.

\begin{lemma}\label{lem:bounded}
The mapping \(T\) is bounded in the following sense:
There exists $M>0$ such that
\[
\|T(f)\|\leq M\|f\|\qquad
(f\in C_0^+(X)).
\]
\end{lemma}

\begin{proof}
Assume, to the contrary, that $T$ is not bounded.
For each $n\in\N$ there exists $f_n\in C_0^+(X)$
such that $\|T(f_n)\|>4^n\|f_n\|$.
Since $T(0)=0$, we have $f_n\neq0$
for all $n\in\N$.
Put $g_n=f_n/\|f_n\|$ for $n\in\N$.
The positive homogeneity of $T$ shows that
$\|T(g_n)\|>4^n$.
Since $\|g_n\|=1$, the series
$g=\sum_{n=1}^\infty g_n/2^n$ converges
in $C_0^+(X)$.
For each $n\in\N$, the order preserving property
of $T$ with $g\geq g_n/2^n$ gives
$T(g)\geq T(g_n/2^n)\geq0$.
Since $T$ is positive homogeneous,
we obtain
\[
\|T(g)\|
\geq\left\|T\left(\frac{g_n}{2^n}\right)\right\|
=\frac{1}{2^n}\|T(g_n)\|
>2^n
\]
for all $n\in\N$.
This contradicts the fact that \(T(g)\in C_0^+(Y)\).
This proves that $T$ is bounded.
\end{proof}

\subsection{The finite intersection property for \(F_y\)}

To construct the correspondence
between the underlying spaces,
we introduce a family of functions
determined by the order isomorphism \(T\).

\begin{definition}
For each \(y\in Y\), define
\[
 F_y=\{f\in C_0^+(X): T(f)(y)>0\}.
\]
\end{definition}

We next prove that every finite family
of supports arising from \(F_y\)
has a non-empty intersection.

\begin{lemma}\label{lem:finite_support}
Let \(y\in Y\). For every finite family \(f_1,\ldots,f_n\in F_y\),
\[
 \bigcap_{i=1}^n \supp(f_i)\ne \emptyset.
\]
\end{lemma}

\begin{proof}
Put
\(g=\bigwedge_{i=1}^n f_i\).
By Lemma~\ref{lem:lattice},
\[
 T(g)(y)
 =
 \min_{1\le i\le n} T(f_i)(y).
\]
Since \(f_1,\ldots,f_n\in F_y\), we have
\(T(f_i)(y)>0\) for all \(i\). Hence \(T(g)(y)>0\).
Thus \(g\ne0\) by Lemma~\ref{lem:zero_disjoint}.
Hence \(\coz(g)\ne\emptyset\),
and consequently \(\supp(g)\ne\emptyset\).

Since \(g\le f_i\) for all \(i\), we have
\( \coz(g)\subset \coz(f_i)\)
and hence
\(\supp(g)\subset \supp(f_i)\)
for all \(i\). Therefore
\(\emptyset\ne \supp(g)\subset \bigcap_{i=1}^n \supp(f_i)\).
\end{proof}

Using the boundedness of \(T\), we now show that if a function belongs to \(F_y\),
a sufficiently small truncation of it still belongs to \(F_y\).
This property will be crucial for extracting compactly supported
functions from \(F_y\).

\begin{lemma}\label{lem:I_y-closed}
Let $y\in Y$ and $f\in C_0^+(X)$.
If $f\in F_y$, then there exists $k\in\N$
such that $(f-1/k)^+\in F_y$.
\end{lemma}

\begin{proof}
Put $f_n=(f-1/n)^+$ and $h_n=f-f_n$
for each $n\in\N$.
Then $f_n,h_n\in C_0^+(X)$ with
$\|h_n\|\leq1/n$.
It follows that
\[
f(x)=f_n(x)+h_n(x)
\leq2\max\{f_n(x),h_n(x)\}
\qquad(x\in X),
\]
and hence $f\leq2(f_n\vee h_n)$
for all $n\in\N$.
The order preserving property
and positive homogeneity
with Lemma~\ref{lem:lattice} give
\[
T(f)\leq2(T(f_n)\vee T(h_n)).
\]
Since $f\in F_y$, we have
$T(f)(y)>0$.
Because $T$ is bounded
by Lemma~\ref{lem:bounded},
there exists $M>0$ such that
$\|T(h_n)\|\leq M\|h_n\|$
for all $n\in\N$.
Choose $k\in\N$ such that
$2M/k<T(f)(y)$.
Since $\|h_k\|\leq1/k$,
we have
\[
T(h_k)(y)
\leq\|T(h_k)\|
\leq M\|h_k\|
\leq\frac{M}{k}
<\frac{T(f)(y)}{2}.
\]
By the preceding inequality,
$T(f)(y)\leq2(T(f_k)\vee T(h_k))(y)$.
These two inequalities show that
$T(f_k)(y)>T(h_k)(y)$.
Indeed, if $T(f_k)(y)\leq T(h_k)(y)$, then
\[
T(h_k)(y)
<\frac{T(f)(y)}{2}
\leq
(T(f_k)\vee T(h_k))(y)
=T(h_k)(y),
\]
a contradiction.
Thus $0<T(f)(y)\leq2T(f_k)(y)$.
This proves $f_k\in F_y$.
\end{proof}

\subsection{Construction of the homeomorphism}

We now prove that the supports
of all functions in \(F_y\)
have a non-empty intersection.
This will be the key ingredient
for analyzing intersections
of zero sets and eventually constructing
the correspondence between the underlying spaces.

\begin{lemma}\label{lem:all_supports_intersect}
For each \(y\in Y\),
\[
 \bigcap_{g\in F_y}\supp(g)\ne \emptyset.
\]
\end{lemma}

\begin{proof}
First we show that there exists
a function in \(F_y\)
with compact support.
Choose \(u\in C_0^+(Y)\)
with \(u(y)>0\), and put
\(f=T^{-1}(u)\).
Then \(T(f)(y)=u(y)>0\), so \(f\in F_y\).

Put \(f_n=(f-1/n)^+\) for \(n\in\N\).
By Lemma~\ref{lem:I_y-closed},
there exists \(k\in\N\) such that
\(f_{k}\in F_y\).

Moreover,
\[
 \supp(f_{k})
 \subset
 \left\{x\in X:f(x)\ge \frac{1}{k}\right\}.
\]
Since \(f_k\in F_y\), we have \(f_k\neq0\)
by Lemma~\ref{lem:zero_disjoint}.
Thus \(\supp(f_k)\) is non-empty.
Since \(f\in C_0(X)\),
the set \(\{x\in X:f(x)\ge 1/k\}\) is compact; hence
\(\supp(f_k)\) is a non-empty compact set.

Let \(g_1,\dots,g_n\in F_y\) be arbitrary.
Applying Lemma~\ref{lem:finite_support}
to \(f_k,g_1,\dots,g_n\), we obtain
\[
\supp(f_k)
\cap
\bigcap_{i=1}^{n}\supp(g_i)
\neq\emptyset.
\]
Hence the family
\[
A=\{\supp(f_k)\cap\supp(g):g\in F_y\}
\]
has the finite intersection property.
Since every member of \(A\) is a closed subset
of the compact set \(\supp(f_k)\),
it follows that
\[
\bigcap_{g\in F_y}
\bigl(\supp(f_k)\cap\supp(g)\bigr)
\neq\emptyset.
\]
Consequently,
\(\bigcap_{g\in F_y}\supp(g)
\neq\emptyset\).
\end{proof}

To continue the construction
of the correspondence between the underlying spaces,
we introduce the following family of functions.
Our next goal is to show that
the associated zero sets
have a non-empty intersection.

\begin{definition}
For \(y\in Y\), define
\[
 I_y=\{f\in C_0^+(X): T(f)(y)=0\}.
\]
For \(f\in C_0^+(X)\), denote its zero set by
\[
 Z(f)=\{x\in X:f(x)=0\}.
\]
\end{definition}

We prove that the family of zero sets
associated with \(I_y\)
has a non-empty intersection.
This will serve as the starting point
for constructing the correspondence
between the underlying spaces.

\begin{lemma}\label{lem:zero_intersection_nonempty}
For each \(y\in Y\),
\[
 \bigcap_{f\in I_y}Z(f)\ne\emptyset.
\]
\end{lemma}

\begin{proof}
By Lemma~\ref{lem:all_supports_intersect},
there exists
\[
 x_0\in \bigcap_{g\in F_y}\supp(g).
\]
We prove that \(f(x_0)=0\) for every \(f\in I_y\).
If \(f=0\), there is nothing to prove.
Assume \(f\neq0\).
Then \(T(f)\ne0\), since \(T(0)=0\).

For each \(n\in\N\) and \(f\in I_y\), put
\[
 u_n=\left(T(f)-\frac1n\right)^+\in C_0^+(Y).
\]
Then
\(\supp(u_n)\subset \left\{z\in Y: T(f)(z)\ge1/n\right\}\).
Since \(T(f)\) is continuous and \(T(f)(y)=0\), the set
\(V_n=\left\{z\in Y: T(f)(z)<1/n\right\}\)
is an open neighborhood of \(y\).

By Urysohn's lemma, there exists
\(v_n\in C_0^+(Y)\) such that
\(v_n(y)>0\) and \(\supp(v_n)\subset V_n\).
Consequently,
\(\supp(u_n)\cap\supp(v_n)=\emptyset\).
Equivalently,
\(u_n\wedge v_n=0\).
By Lemma~\ref{lem:zero_disjoint},
applied to \(T^{-1}\), we have
\(T^{-1}(u_n)\wedge T^{-1}(v_n)=0\).
Put \(f_n=T^{-1}(u_n)\) and
\(g_n=T^{-1}(v_n)\), and then
\(f_n\wedge g_n=0\).
Thus
\(\coz(f_n)\cap \coz(g_n)=\emptyset\).
By the continuity of \(f_n,g_n\),
both \(\coz(f_n)\) and \(\coz(g_n)\)
are open subsets of \(X\).
Therefore
\[
 \coz(f_n)\cap\overline{\coz(g_n)}
=\emptyset.
\]
Since \(v_n(y)>0\), we have \(g_n=T^{-1}(v_n)\in F_y\).
Hence, by the choice of \(x_0\),
\[
 x_0\in \supp(g_n)=\overline{\coz(g_n)}.
\]
Since \(\coz(f_n)\cap\overline{\coz(g_n)}=\emptyset\),
we obtain \(x_0\notin \coz(f_n)\).
Therefore
\(T^{-1}(u_n)(x_0)=f_n(x_0)=0\)
for every \(n\in\N\).
Put
\[
G_{x_0}=\{v\in C_0^+(Y):T^{-1}(v)(x_0)>0\}.
\]
Apply Lemma~\ref{lem:I_y-closed} to the positive homogeneous
order isomorphism
\(T^{-1}:C_0^+(Y)\to C_0^+(X)\).
If \(T(f)\in G_{x_0}\), then there exists \(k\in\N\) such that
\((T(f)-1/k)^+\in G_{x_0}\).
However, since \(u_n=(T(f)-1/n)^+\) and
\(T^{-1}(u_n)(x_0)=0\) for every \(n\in\N\), we have
\(u_n\notin G_{x_0}\) for every \(n\in\N\).
This is a contradiction. Hence \(T(f)\notin G_{x_0}\).
Thus
\(f(x_0)=T^{-1}(T(f))(x_0)=0\)
for all \(f\in I_y\), and so
\(x_0\in \bigcap_{f\in I_y}Z(f)\).
\end{proof}

We now show that the intersection
of the zero sets associated with \(I_y\)
consists of exactly one point.
This enables us to define
the correspondence between
the underlying spaces.

\begin{lemma}\label{lem:zero_intersection_singleton}
For each \(y\in Y\), the set
\[
 \bigcap_{f\in I_y}Z(f)
\]
is a singleton.
\end{lemma}

\begin{proof}
First, this intersection is non-empty
by Lemma~\ref{lem:zero_intersection_nonempty}.
Suppose that this intersection contains
two distinct points \(x_1,x_2\in X\).
Applying Urysohn's lemma,
there exist \(f_1,f_2\in C_0^+(X)\) such that
\[
 f_1(x_1)>0,\qquad
 f_2(x_2)>0,\qquad
 f_1\wedge f_2=0.
\]
By Lemma~\ref{lem:zero_disjoint},
\(T(f_1)\wedge T(f_2)=0\).
Evaluating at \(y\), we obtain
\(\min\{T(f_1)(y),T(f_2)(y)\}=0\).
There exists \(i\in\{1,2\}\) such that
\(T(f_i)(y)=0\).
Then \(f_i\in I_y\).
Since \(x_i\in\bigcap_{f\in I_y}Z(f)\),
we obtain \(x_i\in Z(f_i)\),
and hence \(f_i(x_i)=0\),
a contradiction.
Therefore the intersection
is a singleton.
\end{proof}

By associating each point \(y\in Y\)
with the common zero of the family \(I_y\),
we obtain a mapping from \(Y\) to \(X\).
This mapping will be the essential ingredient
in showing that the order isomorphism
\(T:C_0^+(X)\to C_0^+(Y)\)
is a weighted composition operator.

\begin{definition}\label{def:tau}
By Lemma~\ref{lem:zero_intersection_singleton},
for each \(y\in Y\), the set
\[
 \bigcap_{f\in I_y}Z(f)
\]
consists of exactly one point.
We denote this point by \(\tau(y)\in X\). Thus
\[
 \bigcap_{f\in I_y}Z(f)=\{\tau(y)\}.
\]
This defines a mapping \(\tau:Y\to X\).

Applying the preceding construction to the order isomorphism
\(T^{-1}:C_0^+(Y)\to C_0^+(X)\),
we obtain a mapping
\(\sigma:X\to Y\).
For each \(x\in X\), it is characterized as follows. If
\[
 J_x=\{v\in C_0^+(Y): T^{-1}(v)(x)=0\},
\]
then
\[
 \bigcap_{v\in J_x}Z(v)=\{\sigma(x)\}.
\]
\end{definition}

We next show that the mappings
constructed above are bijective.
This is the first step toward proving
that \(\tau\) is a homeomorphism.

\begin{lemma}\label{lem:tau_bijective}
The mappings \(\tau:Y\to X\) and
\(\sigma:X\to Y\) are bijective
with \(\sigma=\tau^{-1}\).
\end{lemma}

\begin{proof}
Let \(y\in Y,x\in X\).
From the definitions of \(\tau\) and \(\sigma\), we have
\[
f(\tau(y))=0\quad
(f\in I_y),\qquad
v(\sigma(x))=0\quad
(v\in J_x).
\]
For each \(f\in I_y\), we obtain
\(T^{-1}(T(f))(\tau(y))=f(\tau(y))=0\),
and so \(T(f)\in J_{\tau(y)}\).
Thus \(T(f)(\sigma(\tau(y)))=0\).
Since \(f\in I_y\), we have
\(T(f)(y)=0\).
Therefore
\[
T(f)(\sigma(\tau(y)))=T(f)(y)=0
\qquad(f\in I_y).
\]
This shows that \(\sigma(\tau(y))=y\).
Indeed, suppose that
\(\sigma(\tau(y))\neq y\).
Since \(T\) is surjective and
\(C_0^+(Y)\) separates points,
there exists \(f_0\in C_0^+(X)\) such that
\[
T(f_0)(\sigma(\tau(y)))=1,
\qquad
T(f_0)(y)=0.
\]
Then \(f_0\in I_y\) by definition.
The preceding equality shows
\(T(f_0)(\sigma(\tau(y)))=0\),
a contradiction.
Hence \(\sigma\circ\tau\) is the identity
mapping on \(Y\).
By the symmetric argument,
we obtain \(\tau(\sigma(x))=x\).
Thus \(\tau\circ\sigma\) is the identity
mapping on \(X\).
Therefore \(\tau\) and \(\sigma\) are
bijective with \(\sigma=\tau^{-1}\).
\end{proof}

We now describe the correspondence
between the families \(I_y\) and \(J_x\)
in terms of the mapping \(\tau\).

\begin{lemma}\label{lem:zero_equiv}
For every \(y\in Y\),
\(T(I_y)=J_{\tau(y)}\).
That is, for every \(f\in C_0^+(X)\)
and \(y\in Y\),
\[
 T(f)(y)=0
 \quad\Longleftrightarrow\quad
 f(\tau(y))=0.
\]
\end{lemma}

\begin{proof}
If \(v\in T(I_y)\), then \(v=T(f)\)
for some \(f\in I_y\).
Then \(\tau(y)\in Z(f)\)
by the definition of \(\tau(y)\),
and thus \(f(\tau(y))=0\).
We obtain \(v\in J_{\tau(y)}\),
because \(T^{-1}(v)(\tau(y))=f(\tau(y))=0\).
This shows that
\(T(I_y)\subset J_{\tau(y)}\).

Conversely, if \(v\in J_{\tau(y)}\),
then \(\sigma(\tau(y))\in Z(v)\)
by definition.
By Lemma~\ref{lem:tau_bijective},
we obtain \(v(y)=v(\sigma(\tau(y)))=0\).
Thus \(T(T^{-1}(v))(y)=0\),
and hence \(T^{-1}(v)\in I_y\).
Therefore \(v\in T(I_y)\),
proving \(J_{\tau(y)}\subset T(I_y)\).
This completes the proof.
\end{proof}

We show that the correspondence
\(\tau:Y\to X\)
is in fact a homeomorphism.
This property will be crucial
for representing \(T\)
as a weighted composition operator.

\begin{lemma}\label{lem:tau_homeomorphism}
The mapping \(\tau:Y\to X\) is a homeomorphism.
\end{lemma}

\begin{proof}
We first show that \(\tau\) is continuous.
Let \(y\in Y\),
and let \(U\) be an open neighborhood
of \(\tau(y)\) in \(X\).
Since \(X\) is locally compact Hausdorff,
there exists a relatively compact
open neighborhood \(O\) of \(\tau(y)\) such that
\[
 \tau(y)\in O\subset \overline{O}\subset U.
\]
Choose \(f\in C_0^+(X)\) such that
\[
 f(\tau(y))>0,\qquad \supp(f)\subset O\subset U.
\]
By Lemma~\ref{lem:zero_equiv}, \(T(f)(y)>0\). Hence
\(\coz(T(f))\)
is an open neighborhood of \(y\).

For \(z\in\coz(T(f))\), we have \(T(f)(z)>0\).
By Lemma~\ref{lem:zero_equiv},
this implies \(f(\tau(z))>0\). Thus
\[
 \tau(z)\in \coz(f)\subset \supp(f)\subset U.
\]
Hence \(\coz(T(f))\subset \tau^{-1}(U)\),
and \(\tau\) is continuous at \(y\).

The same argument applied to \(T^{-1}\)
shows that the inverse map
\(\sigma=\tau^{-1}\) is continuous.
Therefore \(\tau\) is a homeomorphism.
\end{proof}

\subsection{Locality and determination of values}

The next lemma shows that
local agreement of functions
near \(\tau(y)\)
determines the value of \(T\)
at the point \(y\).
This is the first step toward
constructing the weight function.

\begin{lemma}\label{lem:local_equality}
Let \(f,g\in C_0^+(X)\) and \(y\in Y\).
Suppose that there exists an open neighborhood
\(U\) of \(\tau(y)\) such that
\[
 f|_U=g|_U.
\]
Then
\[
 T(f)(y)=T(g)(y).
\]
\end{lemma}

\begin{proof}
Choose a relatively compact
open neighborhood \(O\) of \(\tau(y)\)
such that
\[
 \tau(y)\in O\subset \overline{O}\subset U.
\]
By Lemmas~\ref{lem:tau_bijective}
and \ref{lem:tau_homeomorphism},
\(\sigma(O)\) is an open neighborhood of \(y\).

Choose \(v\in C_0^+(Y)\) such that
\[
 v(y)>\max\{T(f)(y),T(g)(y)\},
 \qquad
 \supp(v)\subset\sigma(O).
\]
Put
\(h=T^{-1}(v)\in C_0^+(X)\).
We claim that \(\supp(h)\subset \overline{O}\subset U\).
Indeed, let \(z\in X\setminus\overline{O}\).
Since \(\sigma=\tau^{-1}\) is injective, we have
\(\sigma(z)\notin \sigma(\overline{O})\).
Since
\[
 \supp(v)\subset\sigma(O)\subset \sigma(\overline{O}),
\]
it follows that
\(v(\sigma(z))=0\).
By Lemma~\ref{lem:zero_equiv}, applied to \(T^{-1}\), we get
\[
 h(z)=T^{-1}(v)(z)=0.
\]
Thus \(\supp(h)\subset\overline{O}\subset U\).

Since \(f=g\) on \(U\) and \(h=0\) off \(U\), we have
\(f\wedge h=g\wedge h\)
on \(X\). Applying \(T\) and using Lemma~\ref{lem:lattice},
\(T(f)\wedge T(h)=T(g)\wedge T(h)\).
Since \(T(h)=v\), this becomes
\(T(f)\wedge v=T(g)\wedge v\).
Evaluating at \(y\),
\[
 \min\{T(f)(y),v(y)\}=\min\{T(g)(y),v(y)\}.
\]
By the choice of \(v\), this implies \(T(f)(y)=T(g)(y)\).
\end{proof}

To show that \(T\) is a weighted composition operator,
it is important to understand how
local order relations near \(\tau(y)\)
are reflected in the values of \(T\) at \(y\).
The following lemma provides precisely this property.

\begin{lemma}\label{lem:local_order}
Let \(f,g\in C_0^+(X)\) and \(y\in Y\).
Suppose that there exists an open
neighborhood \(U\) of \(\tau(y)\) such that
\[
 f|_U\le g|_U.
\]
Then
\[
 T(f)(y)\le T(g)(y).
\]
\end{lemma}

\begin{proof}
Put \(h=f\vee g\).
The assumption gives
\(h|_U=g|_U\).
By Lemma~\ref{lem:local_equality},
\[
 T(h)(y)=T(g)(y).
\]
Since \(f\le h\),
the order-preserving property of \(T\) gives
\(T(f)\le T(h)\).
Evaluating at \(y\), we get
\(T(f)(y)\le T(h)(y)=T(g)(y)\).
\end{proof}

We are now ready to show that
the order isomorphism \(T\)
admits a weighted composition representation.

\begin{lemma}\label{lem:weighted_pointwise}
For each \(y\in Y\), there exists a constant
\(\alpha(y)>0\) such that
\[
 T(f)(y)=\alpha(y)f(\tau(y))
\]
for all \(f\in C_0^+(X)\).
\end{lemma}

\begin{proof}
Fix \(y\in Y\) and put \(x=\tau(y)\).
There exists \(e\in C_0^+(X)\) such that
\(e(x)>0\).
Let \(f\in C_0^+(X)\).

If \(f(x)=0\), then \(f(\tau(y))=0\),
and Lemma~\ref{lem:zero_equiv} gives \(T(f)(y)=0\).

Assume \(f(x)>0\).
Since \(e(x)>0\),
there exists an open neighborhood
\(O\) of \(x\) such that \(e(z)>0\)
for all \(z\in O\).
The quotient \(f/e\) is continuous on \(O\).
Let \(\varepsilon\) be fixed arbitrarily with
\[
 0<\varepsilon<\frac{f(x)}{e(x)}.
\]
Then there exists an open neighborhood
\(U\subset O\) of \(x\) such that
\[
 \frac{f(x)}{e(x)}-\varepsilon
 \le
 \frac{f(z)}{e(z)}
 \le
 \frac{f(x)}{e(x)}+\varepsilon
\]
for all \(z\in U\).
Since \(e>0\) on \(U\), we obtain
\[
 \left(\frac{f(x)}{e(x)}-\varepsilon\right)e|_U
 \le
 f|_U
 \le
 \left(\frac{f(x)}{e(x)}+\varepsilon\right)e|_U.
\]

By Lemma~\ref{lem:local_order}
and the positive homogeneity of \(T\),
\[
 \left(\frac{f(x)}{e(x)}-\varepsilon\right)T(e)(y)
 \le
 T(f)(y)
 \le
 \left(\frac{f(x)}{e(x)}+\varepsilon\right)T(e)(y).
\]
Letting \(\varepsilon\downarrow 0\), we get
\[
 T(f)(y)=\frac{f(x)}{e(x)}\,T(e)(y)
 =\frac{T(e)(y)}{e(x)}f(\tau(y)).
\]

Define
\[
 \alpha(y)=\frac{T(e)(y)}{e(x)}.
\]
This number is independent
of the particular choice of \(e\);
indeed, applying the above formula
to another admissible function \(e'\) gives
\[
 \frac{T(e')(y)}{e'(x)}
 =
 \frac{T(e)(y)}{e(x)}.
\]
Moreover, \(\alpha(y)>0\).
Indeed, if \(T(e)(y)=0\),
then Lemma~\ref{lem:zero_equiv} would imply
\[
 e(x)=e(\tau(y))=0,
\]
contrary to the choice of \(e\).
Hence \(\alpha(y)>0\).
\end{proof}

\subsection{Proof of the main theorem}

To complete the proof of Theorem~\ref{thm:main},
it remains to establish
the uniqueness of \(\tau\) and \(\alpha\),
as well as the continuity and boundedness
of the weight function \(\alpha\).

\begin{proof}[\textbf{Proof of Theorem~\ref{thm:main}}]
By Lemma~\ref{lem:tau_homeomorphism},
there exists a homeomorphism \(\tau:Y\to X\).
By Lemma~\ref{lem:weighted_pointwise},
there exists a function \(\alpha:Y\to(0,\infty)\)
such that
\[
 T(f)(y)=\alpha(y)f(\tau(y))
\]
for all \(f\in C_0^+(X)\) and \(y\in Y\).

Suppose that there exist a homeomorphism
\(\tau':Y\to X\) and a function
\(\alpha':Y\to(0,\infty)\) such that
\[
T(f)(y)=\alpha'(y)f(\tau'(y))\qquad
(f\in C_0^+(X),\,y\in Y).
\]
If \(\tau(y)\neq\tau'(y)\) for some \(y\in Y\),
then we could choose
\(f_0\in C_0^+(X)\) with \(f_0(\tau(y))=1\)
and \(f_0(\tau'(y))=0\).
By the preceding equalities, we obtain
\[
0<\alpha(y)
=\alpha(y)f_0(\tau(y))
=T(f_0)(y)
=\alpha'(y)f_0(\tau'(y))=0,
\]
which is a contradiction.
Therefore \(\tau=\tau'\), and hence
\(\alpha(y)f(\tau(y))=\alpha'(y)f(\tau(y))\)
for all \(f\in C_0^+(X)\) and \(y\in Y\).
For each \(y\in Y\), choose \(f_1\in C_0^+(X)\)
such that \(f_1(\tau(y))=1\).
Then
\[
\alpha(y)f_1(\tau(y))
=
T(f_1)(y)
=
\alpha'(y)f_1(\tau(y)).
\]
Since \(f_1(\tau(y))=1\), we obtain
\(\alpha(y)=\alpha'(y)\).
This proves the uniqueness of
\(\alpha\) and \(\tau\).

We show that \(\alpha\) is continuous.
Let \(y_0\in Y\), and put \(x_0=\tau(y_0)\).
Choose \(g_0\in C_0^+(X)\) such that
\(g_0(x_0)>0\).
Since \(g_0\circ\tau\) is continuous,
there exists an open neighborhood
\(V\) of \(y_0\) such that
\(g_0(\tau(y))>0\) for all \(y\in V\).
On \(V\), we have
\[
 \alpha(y)=\frac{T(g_0)(y)}{g_0(\tau(y))}.
\]
The numerator \(T(g_0)\) is continuous,
and the denominator \(g_0\circ\tau\)
is continuous and nonzero on \(V\).
Therefore \(\alpha\) is continuous on \(V\).
Since \(y_0\) was arbitrary,
\(\alpha\) is continuous on \(Y\).

We prove that \(\alpha\) is bounded.
Since $T$ is bounded by Lemma~\ref{lem:bounded},
there exists $M>0$ such that
$\|T(f)\|\leq M\|f\|$ for all $f\in C_0^+(X)$.
Let $y\in Y$.
Choose $f\in C_0^+(X)$ with
$f(\tau(y))=\|f\|=1$.
It follows that
\[
\alpha(y)
=\alpha(y)f(\tau(y))
=T(f)(y)
\leq\|T(f)\|\leq M.
\]
Since $y\in Y$ was arbitrary, this proves
that $\alpha$ is bounded.

We prove that there exists \(\delta>0\)
such that \(\alpha(Y)\subset[\delta,\infty)\).
Let \(y\in Y\).
Choose \(u_0\in C_0^+(Y)\) such that
\(u_0(y)=1\).
Substituting \(f=T^{-1}(u_0)\)
into the representation formula for \(T\),
we obtain
\[
1=u_0(y)
=T(T^{-1}(u_0))(y)
=\alpha(y)T^{-1}(u_0)(\tau(y)).
\]
Applying the preceding part of the proof
to the order isomorphism \(T^{-1}\),
we obtain a homeomorphism
\(\sigma:X\to Y\)
and a bounded continuous function
\(\beta:X\to(0,\infty)\)
such that
\[
T^{-1}(u)(x)
=
\beta(x)u(\sigma(x))
\]
for all \(u\in C_0^+(Y)\).
It follows from Lemma~\ref{lem:tau_bijective}
that \(u_0(\sigma(\tau(y)))=u_0(y)=1\).
Combining the preceding equalities gives
\[
1=\alpha(y)T^{-1}(u_0)(\tau(y))
=\alpha(y)\beta(\tau(y)).
\]
Since \(\beta\) is bounded,
\(0<\beta(\tau(y))\leq\|\beta\|<\infty\).
Thus
\[
\frac{1}{\|\beta\|}
\le\frac{1}{\beta(\tau(y))}
=\alpha(y).
\]
Therefore \(\alpha(y)\in[\delta,\infty)\)
for all \(y\in Y\) with \(\delta=1/\|\beta\|\).

Conversely, it is straightforward to verify that
the mapping defined by \eqref{eq:T}
is a
positive homogeneous
 order isomorphism.
\end{proof}

Finally, we show that
an order isomorphism between positive cones
determines the linear order structure
of the whole spaces.
For \(f\in C_0(X)\), define
its positive and negative parts by
\[
 f^+(x)=\max\{f(x),0\},\quad
 f^-(x)=\max\{-f(x),0\} \qquad (x\in X).
\]
Then \(f=f^+-f^-\) and
\(f^+,f^-\in C_0^+(X)\).

\begin{proof}[\textbf{Proof of Corollary~\ref{cor:extension}}]
Let \(T:C_0^+(X)\to C_0^+(Y)\)
be a
positive homogeneous
 order isomorphism.
By Theorem~\ref{thm:main},
there exist a constant \(\delta>0\),
a bounded continuous function
\(\alpha:Y\to[\delta,\infty)\)
and a homeomorphism
\(\tau:Y\to X\)
such that
\[
T(f)(y)=\alpha(y)f(\tau(y))
\qquad
(f\in C_0^+(X),\, y\in Y).
\]

Define a mapping
\(\widetilde T:C_0(X)\to C_0(Y)\)
by
\[
\widetilde T(f)(y)
=
\alpha(y)f(\tau(y))
\qquad
(f\in C_0(X),\, y\in Y).
\]
Then \(\widetilde T\) is a bijective linear map.
Moreover,
\[
f\le g
\quad\Longleftrightarrow\quad
\widetilde T(f)\le \widetilde T(g),
\]
since \(\alpha(y)>0\) for all \(y\in Y\).
Hence \(\widetilde T\) is a linear order isomorphism.

It remains to prove uniqueness.
Suppose that
\(T':C_0(X)\to C_0(Y)\)
is another linear order isomorphism extending \(T\).
For every \(f\in C_0(X)\), we have
\(f=f^+-f^-\).
Since \(T'\) extends \(T\),
\[
T'(f^+)=T(f^+),
\qquad
T'(f^-)=T(f^-).
\]
Therefore,
\[
T'(f)
=
T'(f^+)-T'(f^-)
=
T(f^+)-T(f^-)
=
\widetilde T(f).
\]
Hence \(T'=\widetilde T\).
\end{proof}

\section*{Acknowledgment}

The second author was supported by JST SPRING,
Grant Number JPMJSP2121.
The third author was supported by JSPS KAKENHI
Grant Number JP 25K07028.


\begin{thebibliography}{99}
\bibitem{dong}
Y.~Dong, L.~Li, L.~Moln\'ar and N.-C.~Wong,
\emph{Transformations preserving the norm
of means between positive cones
of general and commutative \(C^*\)-algebras},
J. Operator Theory \textbf{88} (2022), 365--406.

\bibitem{gao}
M.C.~Gao and G.M.~An,
\emph{Maps on Positive Cones of $C^*$-algebras},
Acta Math. Sin. Engl. Ser. \textbf{39} (2023), 2095--2118.

\bibitem{hato}
O.~Hatori and L.~Moln\'ar,
\emph{Isometries of the unitary groups
and Thompson isometries of the spaces
of invertible positive elements in \(C^*\)-algebras},
J. Math. Anal. Appl. \textbf{409} (2014), 158--167.

\bibitem{lemm}
B.~Lemmens, O. van~Gaans and H. van~Imhoff,
\emph{On the linearity of order-isomorphisms},
Canad. J. Math. \textbf{73} (2021), 1099--1133.

\bibitem{moln}
L.~Moln\'{a}r,
\emph{Applications of the automatic additivity of positive homogeneous order isomorphisms between positive definite cones in $C^*$-algebras},
in Function Spaces, Theory and Applications,
Fields Inst. Commun. \textbf{87}, Springer, 2023.

\bibitem{scha}
J.J.~Sch\"affer,
\emph{Order-isomorphisms between cones of continuous functions},
Ann. Mat. Pura Appl. (4) \textbf{119} (1979), 205--230.
\end{thebibliography}
\end{document}